\documentclass[12pt]{article}
\usepackage{geometry}
\usepackage{pdflscape}
\usepackage{graphicx}
\usepackage{subcaption}
\usepackage{amsthm}
\usepackage{amsmath}
\usepackage[T1]{fontenc}
\usepackage[utf8]{inputenc}
\usepackage{amssymb}
\usepackage{booktabs}
\usepackage{amsmath}
\usepackage{blindtext} 
\usepackage[numbers,sort&compress]{natbib}
\usepackage{xcolor}
\usepackage{fancyhdr}
\usepackage{subcaption}
\usepackage{mdframed}
\usepackage{bbm}
\usepackage{hyperref}
\hypersetup{
	citebordercolor=green,
	linkbordercolor=red}
\usepackage{parskip}
\usepackage{tikz}
\usepackage{soul}

\newtheorem{theorem}{Theorem}[subsection]
\newtheorem{proposition}{Proposition}[subsection]

\newtheorem{remark}{Remark}[subsection]
\newtheorem{assumption}{Assumption}[subsection]

\numberwithin{equation}{section}

\title{A Natural Stochastic SIS Model, Analysis of Moments and Comparison of Different Perturbation Techniques}
\author{Berk Tan Perçin\thanks{CNR-IMATI "E. Magenes" - Via A. Corti, 12 - 20133 Milano, Italy. \textbf{e-mail: }berktanpercin@cnr.it}}
\date{\today}

\begin{document}
	\maketitle
	\begin{abstract}
		In this study, a new, natural way of constructing a stochastic Susceptible - Infected - Susceptible (SIS) model is provided to the readers. This new approach is natural in the sense that, the disease transmission rate, $\beta$ will be substituted with a generic, almost surely non-negative one dimensional diffusion. $\beta\geq0$ is an essential property in the deterministic model but generally overlooked in stochastic counterparts, see \cite{graySISpaper, ourSISpaper}. Then, under different conditions on the parameters, the infected population's dynamics are identified, ie. boundedness, extinction and persistence of the infection. It turns out, the new stochastic model agrees with its deterministic version, where the basic reproduction number $R^D_0$ determines the limiting dynamics: Extinction when $R^D_0<1$ and persistence when $R^D_0>1$. 
		Then, a novel analytic technique will be provided to approximate the expectation of any well-behaved function of the infected population, including its moments, by increasing power of correction terms. Such a concept is very useful since the on average dynamics of any stochastic SIS model are not tractable due to its non-linearity. Lastly using the first order correction terms, two different perturbations having the same expectation, \eqref{gray perturbation} performed in \cite{graySISpaper} and the Cox–Ingersoll–Ross (CIR) perturbation proposed in this study, will be compared in terms of their expected effect on the infected population dynamics. This comparison methodology is very useful to provide insight to modelers about the effect of different small perturbations on the overall dynamics of the new model.
	\end{abstract}
	
	
	\section{Introduction}
	While the biological and medical sciences focus primarily on the mechanisms of person-to-person infection, epidemiological modeling examines how these individual-level dynamics influence the broader, complex interactions within populations. Being able to predict such large scale dynamics would provide insight on the necessary measures to protect the society, as we all experienced in Covid-19 pandemic. In addition to these properties, same \emph{infection} models can be used to model how information spreads in a population, making these models very important for evaluating and predicting public opinion as well \cite{info_spread_pap, info_spread_pap2, info_spread_pap3}.
	
	The most popular approaches are either using ordinary differential equations (ODEs) to come up with a system of equations to describe the time evolution of the infection in the population see \cite{ SIRS_pert_pap, 3basicEpidemiologicModels, influenzaA_pap}, or if one wants to assign weights to certain connections rather than using an \emph{on average} infection rates, then using Graph Theory, see \cite{graph_th_pap, paperGraphTheory, graph_th_pap2}.
	
	\subsection{The Deterministic SIS Model}
	This paper's focus will be mostly on the so called "Susceptible Infected Susceptible" (SIS) model, analyzed via a system of PDEs. The SIS model groups the population into two categories \emph{susceptible} and \emph{infected} individuals, where upon contact of a susceptible individual with an infected, the susceptible individual is transferred into the infected category with the disease transmission rate $\beta\in\mathbb{R}^+$. While, the infected individuals get rid of the infection with the rate $\gamma\in\mathbb{R}^+$ to become again susceptible. That means in this model there is no immunity considered. Such models are more useful for various types of infections where individuals seem to experience the infection repeatedly. One such example is the gonorrhea infection where the recovered individuals can get re-infected, implying no apparent immunity observed for this infection type \cite{gonorrhoeaePaper}. If instead, immunity after infection is desired, one can use the so called "Susceptible Infected Recovered" (SIR) model, see \cite{Alberto_Matteo_SIR_paper} for both deterministic and stochastic overview.
	
	The deterministic SIS model is defined as:
	
	\begin{align}\label{deterministic SIS model with S}
		\begin{split}
			\frac{dS^D_t}{dt}=&\gamma I^D_t - S^D_tI^D_\beta,\\
			\frac{dI^D_t}{dt}=&S^D_tI^D_t\beta - \gamma I^D_t,
		\end{split}
	\end{align}
	
	where $S^D_t$ and $I^D_t$ represent the susceptible and infected population at time $t$ respectively. Moreover the initial consitions are given as: $S^D_0=s\geq0$ and $I^D_0=x\geq0$. If one assumes the the total size of the population $N_t$ at time $0$ is known and normalized to $N_0=s+x=1$. Then from the system \eqref{deterministic SIS model with S} it can be seen that $N_t=1$ for $t\geq0$. This means one can express the system \eqref{deterministic SIS model with S} as a single ordinary differential equation (ODE):
	
	\begin{equation}\label{deterministic SIS model}
		\frac{dI^D_t}{dt}=I^D_t(1-I^D_t)\beta - \gamma I^D_t,\quad I_0=x\in(0,1).
	\end{equation}
	
	where $S^D_t:=1-I^D_t$ for $t\geq0$. The equation \eqref{deterministic SIS model} can be solved easily by separation of variables and partial fraction decomposition:
	
	\begin{equation}\label{deterministic infected pop soln}
		I^D_t=\frac{e^{t(\beta-\gamma)}x(\beta-\gamma)}{\beta-\gamma + (e^{t(\beta-\gamma)}-1)x\beta}
	\end{equation}
	
	It turns out, one can find a ratio that totally governs the dynamics of the infection. That ratio is called as the \emph{deterministic reproduction number }$R^D_0:=\frac{\beta}{\gamma}$:
	
	\begin{equation*}
		\lim_{t\to\infty}I_t^D=\begin{cases}
			0,&R^D_0\leq1,\\
			\frac{\beta-\gamma}{\beta},&R^D_0>1.
		\end{cases}
	\end{equation*}
	
	Because this number is crucial in determining if the infected population will \emph{extinct} or will continue to \emph{persist}, it is the main interest of modelers.
	
	\subsection{Stochastic SIS Models in Literature}
	
	A more interesting version of the deterministic SIS model \eqref{deterministic SIS model} is when one or more of the rates perturbed to be a stochastic process instead of being a fixed number. Two notable examples of such papers are presented in Gray et. al.  \cite{graySISpaper} and Lanconelli and Perçin \cite{ourSISpaper}. 
	
	\subsubsection{The perturbation in Gray et. al. \cite{graySISpaper}}
	In Gray et. al. \cite{graySISpaper}, the diease transmission rate $\beta$ perturbed to satisfy:
	
	 \begin{equation}\label{gray perturbation}
	 	\beta dt\to\beta dt + \sigma dB_t.
	 \end{equation}
	 
	 For $\beta, \sigma\geq0$. The new perturbation converts the model \eqref{deterministic SIS model} to the SDE \eqref{gray stochastic SIS Model dI_t}:
	 
	 \begin{equation}\label{gray stochastic SIS Model dI_t}
	 	dI_t = \Big(I_t(1-I_t)\beta - \gamma I_t\Big)dt  + \sigma I_t(1-I_t)dB_t,\quad I_0=x.
	 \end{equation}
	 
	 so the disease transmission rate is chosen to be a normal random variable that spreads around the value $\beta t$ with variance $\sigma^2 t$. This source of randomness provides a more realistic interaction rate in the model since the model doesn't assume a fixed number of interactions all the time. 
	 
	 In the paper \cite{graySISpaper}, the authors analyzed the model and reported the limiting behavior of this model that depends on the stochastic reproduction number $R^S_0=R^D_0 - \frac{\sigma^2}{2\gamma}$. and the amplitude of the perturbation:
	 
	 \begin{itemize}
	 	\item If $R^S_0 < 1$ and $\sigma^2<\beta$ or if $\sigma^2>\max\left\{ \beta, \frac{\beta}{\gamma} \right\}$, then the infection will extinct, that is:
	 	$$\lim_{t\to\infty}I_t=0.$$
	 	\item If $R^S_0 > 1$, then the disease will be persistent, that is:
	 	$$\liminf_{t\to\infty}I_t\leq\xi\leq\limsup_{t\to\infty}I_t.,$$
	 	where $\xi := \frac{1}{\sigma^2}\left( \sqrt{\beta^2-2\sigma^2\gamma} - \beta+\sigma^2\right)$. This means the infected population will oscillate above and below the number $\xi$.
	 \end{itemize}
	 
	Then in the paper, these results are supported with numerical simulations.
	 
	 \subsubsection{The perturbation in Lanconelli and Perçin \cite{ourSISpaper}}
	
	In the paper of Lanconelli and Perçin, unlike of what has been done in Gray et. al, a more formal introduction of perturbation was performed. The problem with the perturbation \eqref{gray perturbation} is that one has to accept the heuristic quantities like the differential of Brownian motion. Instead the paper \cite{ourSISpaper} finds a way to represent the solution of the deterministic infected population \eqref{deterministic infected pop soln} as:
	
	\begin{equation}\label{integral expression of I_t}
		I^D_t = \frac{xe^{\int_{0}^t\beta ds-\gamma t}}{1+x\left( e^{\int_{0}^{t}\beta ds-\gamma t}-1+\int_{0}^{t}e^{\int_{0}^{s}\beta dr - \gamma s}\gamma ds \right)}.
	\end{equation}
	
	The advantage of expression \eqref{integral expression of I_t} is that instead of $\beta$ it depends on the integral of $\beta$ explicitly. This makes the same perturbation well defined since when the substitution $\beta dt \to \beta dt + \sigma dB_t$ is performed on the expression \eqref{integral expression of I_t}, one can just modify the integrals:
	
	$$\int_{0}^{t}\beta ds \to \int_{0}^t\beta ds + \sigma B_t.$$
	
	If this formalism is followed, the paper \cite{ourSISpaper} shows that, by applying It\^o formula to \eqref{integral expression of I_t} the perturbed infected population actually solves the Stratonovich version of the SDE \eqref{gray stochastic SIS Model dI_t}, namely $\sigma I_t(1-I_t)\circ dB_t$. This finding was also cross checked by applying the same perturbation of \eqref{gray perturbation} via the polygonal approximation of the Brownian motion, similar to what has been done in \cite{albert_bernardi}. This way one can formally manipulate $\beta\to\beta + \frac{dB^\pi}{dt}$ where $\pi$ is the mesh of the partition. Letting $\pi\to0$ and applying the Wong-Zakai theorem, the paper shows the Stratonovich SDE is obtained with this method will match the SDE solved by \eqref{integral expression of I_t}.
	
	Surprisingly the dynamics of this new stochastic SIS model is the same as its deterministic version unlike of the one proposed by Gray et. al. Namely, when $R^D\leq1$ the infections extincts and when $R^D>1$ the infection persists.
	
	These results are generalized to an Ornstein-Uhlenbeck (O-U) perturbation: 
	$$\int_{0}^t\beta ds\to\beta t + Y_t,$$
	where $Y_t$ solves the SDE $dY_t=-\alpha Y_t + \sigma dB_t$ with initial condition $Y_0=0$ and same conditions for extinction and persistence are reported. This perturbation is more realistic compared to the perturbation considered in Gray et. al. \cite{graySISpaper} because now one can also adjust the level of variability of the perturbed disease transmission coefficient since O-U process having bounded variance $\frac{\sigma^2}{2\alpha}$ unlike of the variance of Brownian motion that increases linearly with time.
	
	Finally, instead of the O-U perturbation considered, a generic SDE in the form:
	\begin{equation}\label{generic perturbation Z_t}
		dZ_t = b(t,Z_t)dt + \sigma dB_t,\quad Z_0=0.
	\end{equation}
	The function $b:[0,T]\times\mathbb{R}$ is considered to be globally Lipschitz continuous in $z$ and uniformly in $t$. In the paper, authors proved that:
	\begin{itemize}
		\item In order to have extinction of the infection, in addition to having the basic reproduction number $R^D_0\leq1$, one should also have $$\limsup_{t\to\infty}\frac{Z_t}{t}\leq0\quad a.s.$$ 
		\item Similarly, in order to have persistence of the infection, in addition to having the basic reproduction number $R^D_0>1$, one should also have $$0\leq\liminf_{t\to\infty}\frac{Z_t}{t}\leq\limsup_{t\to\infty}\frac{Z_t}{t}<\infty\quad a.s.$$
	\end{itemize}
	
	Which also generalizes the conditions on the O-U or Brownian perturbation. Because for both of these perturbations both conditions are satisfied and $\lim_{t\to\infty}\frac{Y_t}{t}=\lim_{t\to\infty}\frac{B_t}{t}=0$. So for these perturbations considered, the limiting behavior of the stochastic infected population model is governed from the deterministic basic reproduction number $R^D_0$.
	
	\subsection{Approach and organization of this paper}
	Although many of the perturbations done in literature generate useful stochastic models, including the papers \cite{graySISpaper, ourSISpaper}, they have a major drawback. The deterministic SIS model \eqref{deterministic SIS model with S}, by construction needs to have a non-negative disease transmission rate $\beta$ to properly reflect the dynamics of the infection. When $\beta<0$, as seen from the system \eqref{deterministic SIS model with S}, the infected population size decreases by each contact with the susceptible individuals. This dynamics doesn't make sense and the effect of this unrealistic choice is apparent in the deterministic dynamics. However unfortunately it is hidden in the variability introduced to the model in its stochastic versions. Because the perturbations considered in Gray et. al. \cite{graySISpaper} and Lanconelli and Perçin \cite{ourSISpaper} substitute the $\beta$ with their diffusions of interest, at any time $t$, there is a non-zero probability for the perturbed coefficient to be negative. Such a property is not desirable from a modeling point of view and generate additional complexities in parameter estimation. 
	
	This is why in this paper, a new generic perturbation will be applied on the deterministic SIS model's \eqref{deterministic SIS model with S} disease transmission coefficient $\beta$, which was not covered in aforementioned papers and will perturb the $\beta$ to be non-negative stochastic process. We call such perturbations as the \emph{natural} perturbations, preserving the nature of the parameter $\beta$. 
	
	In section \ref{section SIS pertubration} the new generic natural perturbation will be introduced in detail and the properties of this natural stochastic SIS model will be proved. These properties include the boundedness of the solution, identifying the conditions which govern the limiting behavior of the new model, extinction or persistence. 
	
	In section \ref{section diffusion examples}, two one dimensional diffusions are reported to the readers that provide the desired properties. Later the findings of section \ref{section SIS pertubration} are supported via computer simulations by using one of the diffusion processes. 
	
	In section \ref{section on average analysis}, a new method to analytically track the expectation of any function of the stochastic infected population will be provided. The technique relies on the associated Feynman-Kac PDE and Perturbation Theory. The findings of this section will be further supported by computer simulations. 
	
	Lastly in section \ref{section perturbation comparisons}, the technique introduced in section \ref{section on average analysis} will be utilized to show, how using two different perturbations, say $P^{(1)}_t$ and $P^{(2)}_t$, each having the same expectation, $\mathbb{E}[P^{(1)}_t]=\mathbb{E}[P^{(2)}_t],\:\forall t\geq0$ might affect, the expected solution of the stochastic infected population, $\mathbb{E}[\varphi(I_t)|I_0=x, P^{(i)}=y]$ for $i\in\{1,2\}$. In other words, this technique will aid modelers and guide them to choose the most appropriate perturbations for their work. It is because the explicit effect of perturbations on the expected dynamics are hidden in the complexity of the model, but this technique provides an iterative way to reveal them.
	
	\section{Natural SIS Perturbation Approach}\label{section SIS pertubration}
	In this framework our aim is to perturb the disease transmission parameter $\beta\in\mathbb{R}^+\cup\{0\}$ in the ODE \eqref{deterministic SIS model}
	to be a stochastic process $\{Y_t\}_{t\geq0}\geq0$, which satisfies the generic SDE:
	\begin{equation}\label{generic Y_t perturbation}
		dY_t = a(Y_t)dt + b(Y_t)dB_t,\quad Y_0=y\in(0,1).
	\end{equation}
	
	to obtain 
	\begin{equation}\label{perturbed SIS model}
		\frac{dI_t}{dt} = I_t(1-I_t)Y_t-\gamma I_t,\quad I_0=x\in(0,1).
	\end{equation}
	Such a nested way to introduce stochasticity to a coefficient is called stochastic volatility models and very common in finance. One example is the Heston model (see \cite{heston_paper}) where the diffusion coefficient of the price is chosen to satisfy another diffusion process, similar to $Y_t$ we introduced in \eqref{generic Y_t perturbation}. Below we explicitly list the properties that $Y_t$ has to satisfy.
	\begin{assumption}\label{assumptions}
		Throughout this analysis it is assumed that the new perturbation $Y_t$ satisfies the generic SDE \eqref{generic Y_t perturbation} with conditions:
		\begin{itemize}
			\item $a(y),\: b(y)\geq 0$ for $y\in(0,1)$.
			\item $a(0)\geq0, b(0)=0$.
			\item $Y_t$ is ergodic.
		\end{itemize} 		
	\end{assumption}
	
	The reason of having first 2 conditions is to enforce the process to be non-negative. Such a perturbation is more natural for the SIS model because the model itself is constructed on the idea of having positive disease transmission parameter. The last condition however, is required to prove the limiting dynamics as will be seen later in this section.	Now we can move on with proving the conditions boundedness, extinction and persistence of this new model
	
	\subsection{Boundedness Of The Solution}
	\begin{proposition}
		As long as the initial condition, $x\in(0,1)$, the solution of the perturbed model \eqref{perturbed SIS model} $\mathbb{P}(I_t\in[0,1))=1$, $\forall t>0$.
	\end{proposition}
	\begin{proof}
		We start by proving the solution $I_t<1$ for all $t>0$. It is because:
		$$\frac{dI_t}{dt}\xrightarrow{I_t\to 1}-\gamma<0.$$
		Hence the process can never reach $1$.
		
		Moreover because when $I_t=0$ the right hand side of the equation \eqref{perturbed SIS model} becomes 0 so the solution can't take negative values.
	\end{proof}
	
	\subsection{The Extinction of the Infection}
	\begin{theorem}\label{Theorem extinction}
		Let the assumptions in \ref{assumptions} be in force. Moreover let $\mathbb{E}[Y_\infty]$ denote the expectation of the process $Y_t$ with respect to the invariant measure. Define $R^S_0:=\mathbb{E}[Y_\infty]/\gamma$. As long as 
		$$R^S_0<1,\quad or \quad\mathbb{E}[Y_\infty]<\gamma,$$
		the infection extincts.
	\end{theorem}
	\begin{proof}
		Define $G(x):=\ln\left(\frac{x}{1-x}\right)$ for $x\in[0,1)$, the differential of $G(x)$ is:
		\begin{align*}
			d\Big( G(I_t) \Big) &= \frac{1}{I_t(1-I_t)}dI_t=Y_tdt - \frac{\gamma}{1-I_t}dt\leq\left( Y_t - \gamma \right)dt,
		\end{align*}
		because $1/(1-x):[0,1)\to[1,\infty)$. Integrating both sides from $0$ to $t$ yield:
		\begin{align*}
			G(I_t)&\leq\ln\left(\frac{x}{1-x}\right)+\int_0^tY_sds -\gamma t\\
			\Rightarrow\limsup_{t\to\infty}\limits\frac{1}{t}G(I_t)&\leq \limsup_{t\to\infty}\limits\frac{1}{t}\ln\left(\frac{x}{1-x}\right)+\limsup_{t\to\infty}\limits\frac{1}{t}\int_0^tY_sds - \gamma\\
			&=\limsup_{t\to\infty}\limits\frac{1}{t}\int_0^tY_sds - \gamma = \mathbb{E}[Y_\infty]-\gamma.
		\end{align*}
		Where the last equality is true because $Y_t$ is assumed to be ergodic in assumption \ref{assumptions}. Because the right hand side of the last equality is negative, it implies that:
		$$\lim_{t\to\infty}I_t=0,$$
		hence extinction.
	\end{proof}
	\begin{remark}
		The intuition is, in the large time limit, the extinction will occur if the disease transmission coefficient is on average less than the recovery rate with respect to the stationary distribution.
		\label{remark extinction of infection}
	\end{remark}
	
	\subsection{The Persistance Of The Infection}
	
	The main proof methodology is from \cite{graySISpaper} and \cite{ourSISpaper}.
	\begin{theorem}\label{persistance theorem}
		Consider the function $f(x):=\frac{-1}{1-x}$ and let the assumptions in \ref{assumptions} be in force. Similarly define $R^S_0:=\mathbb{E}[Y_\infty]/\gamma$. As long as $$R^S_0>1,\quad or \quad\mathbb{E}[Y_\infty]>\gamma,$$
		then the infection persists. That is
		$$\liminf_{t\to\infty}I_t\leq \mathtt{I}^*\quad and \quad\limsup_{t\to\infty}I_t \geq \mathtt{I}^*,$$
		where $\mathtt{I}^*:=f^{-1}\left(\frac{-\mathbb{E}[Y_\infty]}{\gamma}\right)$.
	\end{theorem}
	\begin{proof}
		Note that $f:[0,1)\to[-1,-\infty)$ and $f$ is a monotone decreasing function, we start by the following assertion:
		
		Suppose $\liminf_{t\to\infty}I_t \leq \mathtt{I}^*$ is wrong, then $\exists\:\epsilon>0$ such that:
		$$\mathbb{P}(\Omega_1)>\epsilon,\quad \textnormal{where}\quad \Omega_1:=\left\{\liminf_{t\to\infty}I_t\geq \mathtt{I}^*+\epsilon\right\}.$$
		Therefore, for all $\omega\in\Omega_1$, there exists $T(\omega)\geq0$ such that
		$$I_t\geq \mathtt{I}^*+\epsilon,\quad\textnormal{for all }t\geq T(\omega).$$
		The monotonicity of $f$ yields:
		\begin{equation}\label{f upperbound}
			f(I_t)\leq f(\mathtt{I}^*+\epsilon),\quad\textnormal{for all }t\geq T(\omega).
		\end{equation}
		When It\^o is applied to the function $G(I_t)$ as done in Theorem \ref{Theorem extinction}:
		\begin{align*}
			\limsup_{t\to\infty}\frac{1}{t}\ln\left(\frac{I_t}{1-I_t}\right)&\leq\limsup_{t\to\infty}\frac{1}{t}\ln\left(\frac{x}{1-x}\right) +\limsup_{t\to\infty}\frac{1}{t}\int_0^tY_sds +\gamma\limsup_{t\to\infty}\frac{1}{t}\int_0^tf(I_s)ds,\\
			&=\mathbb{E}[Y_\infty]+\gamma\limsup_{t\to\infty}\frac{1}{t}\int_0^tf(I_s)ds,\\
			&\leq\mathbb{E}[Y_\infty]+\gamma\left[ \limsup_{t\to\infty}\frac{1}{t}\int_0^Tf(I_s)ds+f(\mathtt{I}^*+\epsilon)\limsup_{t\to\infty}\frac{t-T}{t} \right],\\
			&=\mathbb{E}[Y_\infty]+\gamma f(\mathtt{I}^*+\epsilon)<0.
		\end{align*}
		The last inequality is due to the fact that $f(\mathtt{I}^*+\epsilon)< f(\mathtt{I}^*)=\frac{-\mathbb{E}[Y_\infty]}{\gamma}$. Which is a contradiction because this means 
		$$\lim_{t\to\infty}I_t=0.$$
		Then the first assertion has to be true, which is:
		$$\liminf_{t\to\infty}I_t\leq \mathtt{I}^*.$$
		The first inequality is therefore proven.
		
		In order to prove the other part assume the claim $\limsup_{t\to\infty}I_t\geq \mathtt{I}^*$ is wrong. This means $\exists\:\epsilon>0$ such that:
		$$\mathbb{P}(\Omega_2)>\epsilon\quad \textnormal{where} \quad\Omega_2:=\left\{ \limsup_{t\to\infty}I_t\leq \mathtt{I}^*-\epsilon \right\}.$$
		
		Therefore for all $\omega\in\Omega_2$, $\exists\:T(\omega)\geq0$ such that
		$$I_t\leq \mathtt{I}^*-\epsilon,\quad \textnormal{for all }t\geq T(\omega).$$
		Moreover, due to the monotonicity of $f$, one obtains:
		$$f(I_t)\geq f(\mathtt{I}^*-\epsilon)\quad\textnormal{for all }t\geq T(\omega).$$
		Now similarly using the differential of $G(I_t)$ one can write down the inequality:
		\begin{align*}
			\liminf_{t\to\infty}G(I_t)&\geq \liminf_{t\to\infty}\frac{1}{t}\int_0^tY_sds + \liminf_{t\to\infty}\frac{1}{t}\gamma\int_0^tf(I_s)ds\\
			&\geq\mathbb{E}[Y_\infty]+\gamma f(\mathtt{I}^*-\epsilon)>0.
		\end{align*}
		Because $f(\mathtt{I}^*-\epsilon)>f(\mathtt{I}^*)=\frac{-\mathbb{E}[Y_\infty]}{\gamma}$, the last term on the right hand side is positive. Leading to
		$$\lim_{t\to\infty}I_t=1,$$
		which is a contradiction. So the first assertion has to be true:
		$$\limsup_{t\to\infty}I_t\geq \mathtt{I}^*.$$
	\end{proof}
	\begin{remark}
		Note that the condition $\mathbb{E}[Y_\infty]>\gamma$ is necessary in order to find the number $\mathtt{I}^*$ in interval $(0,1)$. Because if $\frac{\mathbb{E}[Y_\infty]}{\gamma}>1$, then $-\frac{\mathbb{E}[Y_\infty]}{\gamma}<-1$ and we can find $\mathtt{I}^*=f^{-1}\left(-\frac{\mathbb{E}[Y_\infty]}{\gamma}\right)$ such that $\mathtt{I}^*\in(0,1)$.
		
		The intuition is that, in the large time limit, the persistance will take place if the disease transmission coefficient is on average more than the recovery rate with respect to the stationary distribution.
		\label{remark persistence of infection}
	\end{remark}
	
%
%
	
	\begin{remark}
		Note that the results of theorems \ref{Theorem extinction} and \ref{persistance theorem} does not change based on the amplitude of the noise $\sigma$. It makes this model useful because one is able to predict the dynamics of the infection without the need to infer the amplitude of the noise. Such a property was absent in the model presented in \cite{graySISpaper}.
		\label{remark no dependence on sigma}
	\end{remark}
	
	\section{Example Diffusions Satisfying Assumptions \ref{assumptions}}\label{section diffusion examples}
	Here some one dimensional diffusions will be listed which satisfies the assumptions \ref{assumptions}.
		\subsection{Stochastic Logistic Equation}
	The Stochastic Logistic Equation is obtained when the functions $a(y)$ and $b(y)$ are chosen as:
	$$a(y)=y(a-by)\quad \textnormal{and}\quad b(y)=\sigma y,$$
	where $a,b,\sigma>0$. This choice now reads as:
	$$dY_t=Y_t(a-bY_t)dt + \sigma Y_tdB_t,\qquad Y_0=y\in(0,1),$$
	for our perturbation.
	
	Similar to before, $\{Y_t\}_{t>0}$ is a suitable perturbation candidate because it is almost surely positive for $t>0$, \cite{stochastic_logistic_paper}. Moreover, when $2a>\sigma^2$, then $\{Y_t\}_{t\geq0}$ is ergodic and has a stationary distribution as again, the gamma distribution \cite{stochastic_logistic_paper} with parameters: $\lambda:=2a/\sigma^2-1$ and $\omega:=2b/\sigma^2$, then the stationary distribution can be written as:
	$$\rho_{stat}(x,a,b,\sigma) = \frac{\omega^{\lambda}}{\Gamma(\lambda)}x^{\lambda-1}e^{-\omega x}.$$
	with the expectation of the stationary distribution being:
	$$\mathbb{E}[Y_\infty]=\frac{2a-\sigma^2}{2b}.$$
	
	Where similarly this identity can be substituted in all of the previous results too.	
	\subsection{The Cox–Ingersoll–Ross (CIR) Model}
	The CIR model corresponds to our generic perturbation when 
	$$a(y) = a(b-y)\quad\textnormal{and}\quad b(y)=\sigma \sqrt{y},$$
	where $a,b,\sigma \in\mathbb{R}^+.$
	 
	So the equation \eqref{generic Y_t perturbation} can be chosen to be:
	\begin{equation}
		dY_t = a(b-Y_t)dt + \sigma \sqrt{Y_t}dB_t,\qquad Y_0=y\in(0,1).
		\label{CIR process eqn}
	\end{equation}
	
	It is well known in literature that, when $2ab/\sigma^2>1$, $\{Y_t\}_{t>0}$ that is, strictly positive for all $t>0$ and ergodic, see \cite{CIR_model_ergodicity, CIR_model_positivity}. These properties make the CIR model a very suitable candidate for perturbing the SIS model in a natural way and analyze it robustly. Let $\omega:=2a/\sigma^2$ and $\lambda:=2ab/\sigma^2$, then the stationary distribution of $Y_\infty$ is the gamma distribution \cite{CIR_model_ergodicity}:
	$$\rho_{stat}(x;a,b,\sigma) = \frac{\omega^\lambda}{\Gamma(\lambda)}x^{\alpha-1}e^{-\omega x}.$$
	which makes $\mathbb{E}[Y_\infty]=b$. So all of the previous results can be substituted with $b$ if the chosen perturbation is the CIR process.
	
	\begin{remark}
		When one chooses $y=b=\beta$, then the expectation of the CIR process is:
		$$\mathbb{E}[Y_t] = ye^{-at} + b(1-e^{-at})=\beta,\quad\forall t>0.$$
		
		In other words this way, one can have the deterministic parameter value $\beta$ as the expectation of the perturbed process. This case will be investigated deeper in the following sections.
		\label{CIR model fixed expectation remark}
	\end{remark}
	
	\subsection{The simulation Results with CIR perturbation in model \eqref{perturbed SIS model}}
	It would be interesting to use a perturbation of the SIS model with a positive process that has constant expectation. For this reason the CIR model will be utilized with 2 different cases:
	\subsubsection{The $R^S_0<1$ case:}
	\begin{figure}[h!]
		\centering
		\includegraphics[width=0.7\textwidth]{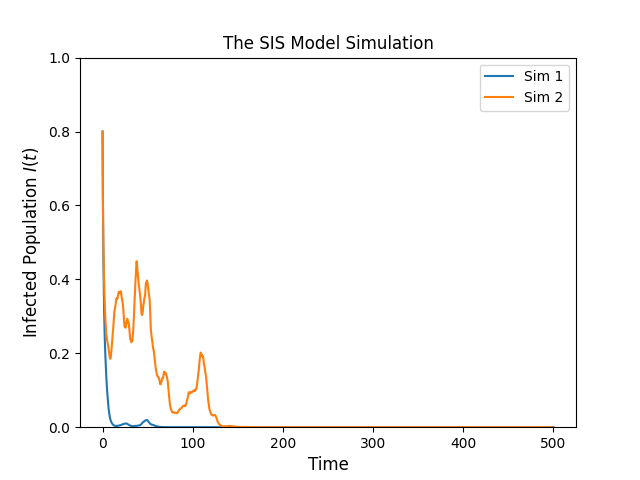}
		\caption{The results of two simulations of the model \eqref{perturbed SIS model} with parameters: $\beta=y=b=0.89$, $\gamma = 0.92$, $\sigma=0.1$, $a=0.05$ and $x=0.8$. Note that $R^S_0 < 1$,  $\mathbb{E}[Y_t]=0.89,\:\forall t>0$, due to remark \ref{CIR model fixed expectation remark} and $2ab/\sigma^2 = 8.9>1$ ensuring positivity of the perturbation.}
		\label{figure SIS extinction}
	\end{figure}
	
	From figure \ref{figure SIS extinction}, one can see the extinction of the infection, due to the disease transmission coefficient's on average value $\mathbb{E}[Y_t]$ being smaller than the recovery rate $\gamma$. 
	
	\subsubsection{The $R^S_0>1$ case:}
	\begin{figure}[h!]
		\centering
		\includegraphics[width=0.7\textwidth]{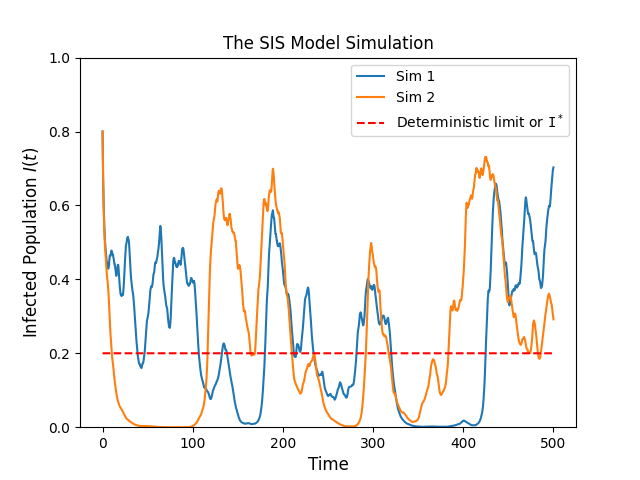}
	\caption{The figure showing the results of two simulations with parameters: $\beta=y=b=0.5$, $\gamma = 0.4$, $\sigma=0.1$, $a=0.05$ and $x=0.8$. Note that $R^S_0 = 1.5>1$, $\mathbb{E}[Y_t]=0.5,\:\forall t>0$, due to remark \ref{CIR model fixed expectation remark} and $2ab/\sigma^2 = 5 >1$ ensuring positivity of the perturbation. The deterministic limit or $\mathtt{I}^*=f^{-1}(-5/4)=0.2$.}
	\label{figure SIS persistence}
	\end{figure}
	
	In figure \ref{figure SIS persistence}, the persistent dynamics of the CIR perturbation is reported. Due to the theorem \ref{persistance theorem}, it is found that the infected population will oscillate above and below the point $\mathtt{I}^*$, which corresponds to the limit of the deterministic SIS model \eqref{deterministic SIS model}.
	
%
%
	
	\section{On Average Analysis of the Model \eqref{perturbed SIS model}}\label{section on average analysis}
	It is accustomed in literature to come up with a perturbation method to the SIS model and then provide the limiting conditions for its long term behavior. In this section instead, we want to trace the following types of question: \emph{How does this perturbation change the model \eqref{deterministic SIS model} on average?} In other words, our main aim is to come up with an analytical technique to analyze the dynamics of the moments of the model \eqref{perturbed SIS model}. In the literature, the study \cite{ccetin2025approximate}, proposed an iterative way for approximating the moments of a general stochastic logistic equation in the form:
	$$dX_t:=(AX_t-\delta X_t^r)dt + \sigma X_t dB_t,$$
	where $A\in\mathbb{R}$, $\delta\geq0$ and $r,\sigma>0$. However because all of the stochastic SIS models considered in this paper displays a non-linear diffusion term, another approach will be necessary for our analysis.
	
	Another method to answer such question, albeit being a crude method in literature, is to use mean-field approximation, which is very common in physics \cite{meanfield_isingmodel_paper}. This method is used in SIS models too, especially in complex networks framework \cite{meanfield_SIS_graph_paper} where this approach reduces the complexity of the interaction term between nodes. The logic is similar here, note that $S_t = 1-I_t$, then the term $I_tS_tY_t$ in expression \eqref{perturbed SIS model} is quantifying a complex interaction term considering all individuals. Using the mean field approximation, instead of considering all the randomness effects from all individuals, one only considers an on average effect of disease transmission in total and write:
	
	\begin{equation}
		\mathbb{E}[I_t(x,Y_t,\gamma)] \approx I_t(x,\mathbb{E}[Y_t], \gamma).
		\label{mean field SIS approximation}
	\end{equation}
	
	Of course the relation \eqref{mean field SIS approximation} is not true given the solution $I_t$ is non-linear but it simplifies the analysis a lot and provide a somewhat approximation of the solution. This approach may seem like a step back, because now from the stochastic model we developed in section \ref{section SIS pertubration} we jumped back to a deterministic model again with $Y_t$ being replaced by its expectation $\mathbb{E}[Y_t]$. So according to this approach, the deterministic version of this model, is an approximation of the stochastic one.
	
	In the following sections we provide a way of evaluating \emph{correcting terms} to the crude mean-field approximation within a formal framework. In section \ref{section feynman kac} we convert our SDE \eqref{perturbed SIS model} to the associated Feynman-Kac PDE (fkPDE), in section \ref{section perturbation theory} the Perturbation Theory will be used to convert the fkPDE to an infinite system of coupled PDEs and the method to solve this infinite system will be presented. Lastly in section \ref{section results correction terms} the results will be reported with the first correction term.
	
	\subsection{The Associated Feynman-Kac Equation}\label{section feynman kac}
	
	In order to do so we rewrite the process \eqref{generic Y_t perturbation} by introducing another parameter $c\in[0,1)$ in a slightly different format:
	
	\begin{equation}\label{perturbation Y_t scaled with c}
		dY_t := c\tilde{a}(Y_t)dt + \sqrt{c}\tilde{b}(Y_t)dB_t,\quad Y_0=y\in(0,1).
	\end{equation}
	
	Where $a(y)=c\tilde{a}(y)$ and $b(y)=\sqrt{c}\tilde{b}(y)$ in the expression \eqref{generic Y_t perturbation}. The advantage of this new format will be apparent when the Feynman-Kac formula is utilized for the perturbed SIS model \eqref{perturbed SIS model}. Define $\nu(t, I_t, Y_t):=\varphi(I_t(Y_t))$ for a sufficiently well-behaved function $\varphi:[0,1]\to\mathbb{R}$, apply It\^o formula on $\nu(t, I_t, Y_t)$ and then set $u(t,x,y):=\mathbb{E}[\nu(t,I_t, Y_t)|I_0=x, Y_0=y]$ or equivalently $u(t,x,y)=\mathbb{E}[\varphi(I_t)|I_0=x, Y_0=y]$. Then $u(t,x,y)$, solves:
	
	\begin{equation}\label{feynman kac formula for u(t,x,y)}
		\frac{\partial u}{\partial t} = \Big(yx(1-x)-\gamma x\Big)\frac{\partial u}{\partial x}+c\Big( \frac{1}{2}\tilde{b}^2(y)\frac{\partial^2 u}{\partial y^2}+\tilde{a}(y)\frac{\partial u}{\partial y} \Big),\quad u(0,x,y) = \varphi(x).
	\end{equation}
	
	This is the associated Feynman-Kac equation of the perturbed SIS model with respect to the diffusion \eqref{perturbation Y_t scaled with c}. One can see that the role of the scaling of $c$ in \eqref{perturbation Y_t scaled with c} is to appropriately add an extra term as the perturbation in \eqref{feynman kac formula for u(t,x,y)}. If one would instead consider $dY_t:=c\tilde{a}(Y_t)dt+c\tilde{b}(Y_t)dB_t$, the effect of the diffusion would be scaled with $c^2$ in \eqref{feynman kac formula for u(t,x,y)}. So in order to give similar importance to both the drift and diffusion of the perturbation in $u(t,x,y)$, the appropriate scaling in expression \eqref{perturbation Y_t scaled with c} is chosen.
	
	\subsection{The Perturbation Theory Approach}\label{section perturbation theory}
	Suppose the solution of the equation \eqref{feynman kac formula for u(t,x,y)} is in the power series form:
	$$u(t,x,y)=\sum_{n\geq0}u_n(t,x,y)c^n,$$
	for a series of functions $u_n$. When this power series expansion is substituted in equation \eqref{feynman kac formula for u(t,x,y)} and all terms with the same order of $c$ is collected together, one obtains a series of equations:
	\begin{align}
		\begin{split}
			\frac{\partial u_0}{\partial t} = \Big(yx(1-&x) - \gamma x\Big)\frac{\partial u_0}{\partial x},\\
			\frac{\partial u_1}{\partial t} = \Big( yx(1-x)-\gamma x \Big)\frac{\partial u_1}{\partial x} &+ \Big(\frac{1}{2} \tilde{b}^2(y)\frac{\partial^2 u_0}{\partial y^2}+\tilde{a}(y)\frac{\partial u_0}{\partial y}\Big),\\
			\frac{\partial u_2}{\partial t} = \Big( yx(1-x)-\gamma x \Big)\frac{\partial u_2}{\partial x} &+ \Big(\frac{1}{2} \tilde{b}^2(y)\frac{\partial^2 u_1}{\partial y^2}+\tilde{a}(y)\frac{\partial u_1}{\partial y}\Big),\\
			&\vdots
		\end{split}
		\label{eqn system of u with correction terms}
	\end{align}
	with the initial conditions $u_0(0,x,y)=\varphi(x)$ and $u_n(0,x,y)=0$ for $n\geq1$. Note that the system of partial differential equations (PDEs) are solvable because once the solution of the previous line is found,  it only generates the non-homogeneous term on the next one and the first state is just the solution of the deterministic SIS model. This brings us to the following proposition:
	
	\begin{proposition}
		The solution to the system \eqref{eqn system of u with correction terms} is given as:
		\begin{equation}
			u_n(t,x,y) = \int_{0}^{t}\left[ 	\frac{1}{2}\tilde{b}^2(y)\frac{\partial^2u_{n-1}(s, I^D_{t-s}(y, x), y)}{\partial y^2}+\tilde{a}(y)\frac{\partial u_{n-1}(s, I^D_{t-s}(x,y), y)}{\partial y} \right]ds,
			\label{solution to correction terms system}
	\end{equation}
		for $n\geq1$ and: 
		\begin{equation}
			u_0(t,x,y)=\varphi(I^D_{t}(x,y))
			\label{u0 solution}
		\end{equation}
		where $I^D_t$ is the solution of the deterministic SIS model reported in \eqref{deterministic infected pop soln}.
		\label{proposition solving system correction terms}
	\end{proposition}
	\begin{proof}
		Proving the equality of $u_0$ in expression \eqref{u0 solution} is a direct application of the Feynman-Kac formula on the $\varphi(I^D_t)$, where $I^D_t$ is the solution to the deterministic SIS model \eqref{deterministic SIS model} with initial conditions $I^D_0=x$ and $Y_0=y$.
		
		When one wants to solve for $u_1(t,x,y)$, it is shown that the solution of the previous state $u_0$ only enters as a non-homogeneous term in the system \eqref{eqn system of u with correction terms}. Let $\alpha(x,y):=yx(1-x)-\gamma x$ and $f(t,x,y)$ be the non-homogeneous term $\frac{1}{2} \tilde{b}^2(y)\frac{\partial^2 u_0(t,x,y)}{\partial y^2}+\tilde{a}(y)\frac{\partial u_0(t,x,y)}{\partial y}$, then the problem is converted to the form
		\begin{equation}\label{from u1 to generic un}
			\frac{\partial u_1}{\partial t} = \alpha(x,y)\frac{\partial u_1}{\partial x} + f(t,x,y),\quad u_1(0,x,y)=0.
		\end{equation}
		Following the method of characteristics, we want to solve the PDE along a certain path $x(s)$ so that, along that path the PDE acts as a simple ODE. In order to do so suppose that path satisfies the relation:
		\begin{equation}\label{method of characteristics path x(t)}
			\frac{dx(s)}{ds}=-\alpha(x(s),y)=-yx(s)(1-x(s))+\gamma x(s),\quad\textnormal{where }s<t\textnormal{ and }x(t)=x.
		\end{equation}
		
		The terminal condition $x(t)=x$ satisfies that at time $t$, solution $u_1(t,x(t),y)$ will depend on the variable $x$ as desired. Note equation \eqref{method of characteristics path x(t)} is nothing but the deterministic SIS ODE \eqref{deterministic SIS model} with a terminal condition $x(t)=x$ instead of an initial one and an additional minus sign. The solution can be found similarly as:
		\begin{equation}\label{x(s) explicit solution}
			x(s)=I^D_{t-s}(x,y)=\frac{x(e^{(t-s)(y-\gamma)}(y-\gamma)}{y-\gamma-xy\big( 1-e^{(t-s)(y-\gamma)} \big)}
			\quad\textnormal{for }0<s<t. 
		\end{equation}
		When this relation is substituted to the PDE, one obtains a single ODE and the solution of this ODE is given as:
		\begin{equation}\label{from u1 to generic un explicit solution}
			u_1(t,x,y)=\int_{0}^tf\big(s,I^D_{t-s}(x,y),y\big)ds.
		\end{equation}
		Finally note that once $u_1$ is evaluated the same process can be applied to $u_2$ and so on. The results obtained in expressions \eqref{from u1 to generic un}, \eqref{method of characteristics path x(t)} and \eqref{x(s) explicit solution} does not depend on the $u_1$ at all, it only depends on the solution on the previous state $u_0$ via the non-homogeneous term $f(t,x,y)$. Hence the formula \eqref{from u1 to generic un explicit solution} is generic and can be applied to any $n$ as \eqref{solution to correction terms system}. This completes the proof.
	\end{proof}
	
	\begin{remark}\label{remark small correction}
		Note that the first line corresponds to the solution of the deterministic SIS model. This means, due to the form of the solution being: $$u(t,x,y)=u_0(t,x,y) + cu_1(t,x,y)+c^2u(t,x,y)+\dots$$ 
		all the terms with with positive power of $c$ can be viewed as a correction term that modifies the deterministic solution $u_0$. This way of expanding a solution of a single PDE \eqref{feynman kac formula for u(t,x,y)} to a coupled system of PDEs \eqref{eqn system of u with correction terms} in increasing orders of "$c$" is called the Perturbation Theory and it is common to find some applied examples in finance (see \cite{perTh_finance_paper2, pertTh_finance_paper}). The advantage of this approach is when the unperturbed model is analytically solvable, one can build the perturbed solution based on modifying the unperturbed one via the correction terms, as we have here.
		
		Moreover because $c$ only enters with increasing powers, for a small $c<<1$, the first few terms should be enough to show the effect of correction. So for $c<<1$, $$\mathbb{E}[\varphi(I_t)|I_0=x, Y_0=y]=u(t,x,y)\approx u_0(t,x,y)+cu_1(t,x,y)$$.
	\end{remark}
\subsection{Results on Expectation of the Model \eqref{perturbed SIS model} with the CIR Perturbation}\label{section results correction terms}
Take $\varphi(x)=x$ in equation \eqref{feynman kac formula for u(t,x,y)} and this way $u^{(1)}(t,x,y)=\mathbb{E}[I_t|I_0=x, Y_0=y]$, we use $u^{(1)}$ to denote $\varphi(x)=x$ choice. Then the form of the solution we are searching form read as $u^{(1)}(t,x,y)=\sum_{n\geq0}u^{(1)}_n(t,x,y)c^n$.

In this section we are interested in the first correction term. In order to come up with it, the system \eqref{eqn system of u with correction terms} is solved via the proposition \ref{proposition solving system correction terms}. Because evaluating the integral of first and second or derivatives of $u^{(1)}_0$ are quite demanding by hand, these calculations were performed in symbolic derivation and integration in Wolfram. 

	Then in order to test the accuracy of the new corrected expectation, the stochastically perturbed model \eqref{perturbed SIS model} with the CIR process is simulated for 1500 sample paths. Then their results are averaged to yield and expectation estimate, when these results are plotted together with $u^{(1)}_0$ and $u^{(1)}_0+ c u^{(1)}_1$ the results obtained are presented in figure \ref{figure average and correction terms}.
	
	\begin{figure}[h]
		\centering
		\begin{subfigure}{0.49\textwidth}
			\centering
			\includegraphics[width=\linewidth]{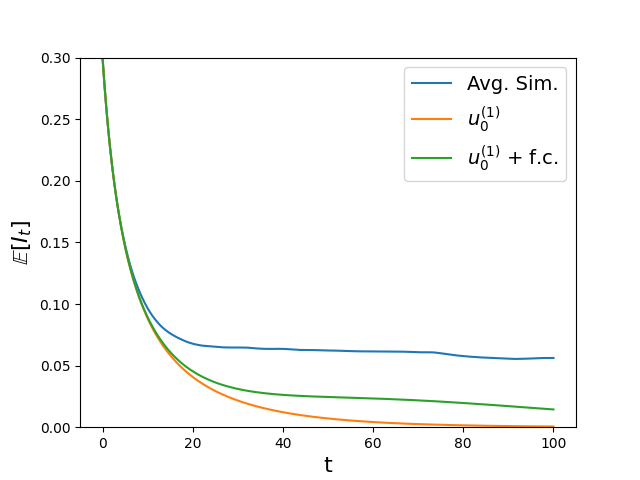}
			\caption{Extinction of infection.}
		\end{subfigure}
		\hfill
		\begin{subfigure}{0.49\textwidth}
			\centering
			\includegraphics[width=\linewidth]{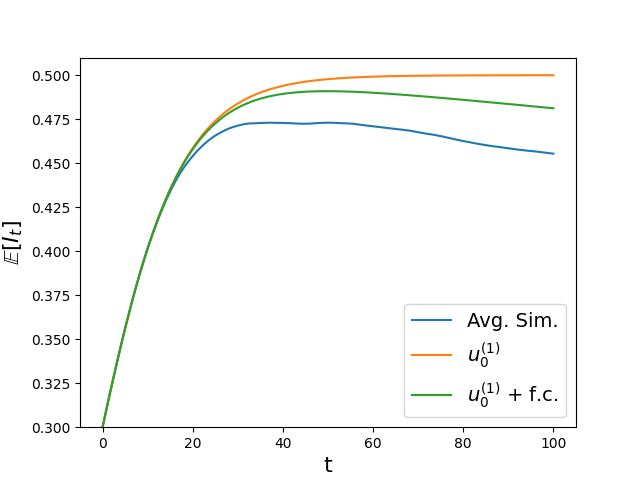}
			\caption{Persistence of infection.}
		\end{subfigure}
		\caption{The plot showing the average of simulation sample paths in blue, deterministic or mean-field solution $u^{(1)}_0(t)$ in orange and $u^{(1)}_0(t)$ plus the first correction term ($u^{(1)}_1(t)$) in green. The parameters are (a): $c=0.1$, $x=0.3$, $y=b=0.45$, $\gamma=0.5$, $\sigma=0.063$, $a=0.02$ and (b): $c=0.1$, $x=0.3$, $y=b=0.2$, $\gamma=0.1$, $\sigma=0.032$, $a=0.02$. The number of averaged simulations is $1500$.}
		\label{figure average and correction terms}
	\end{figure}
		
	Note for the CIR perturbation, the $u^{(1)}_0$ in figure \ref{figure average and correction terms} corresponds to the mean-free approximation \eqref{mean field SIS approximation}, since the expectation of the CIR model with the selected parameters is only the initial condition $y$. One can see that the mean-field method approximates the solution to some extend but the expectation with the first correction term performs better.
	
	\subsection{Results on Variance of the Model \eqref{perturbed SIS model} with the CIR Perturbation}\label{section variance results correction terms}
	One can see that the equation \eqref{feynman kac formula for u(t,x,y)} holds for any well-behaved function $\varphi(x)$. This makes the fkPDE approach very powerful because instead of only focusing on the expectation of the solution, one can actually focus on expectation of any function of the solution. Just to give an example here we will consider $\varphi(x)=x^2$. Using the second moment and the first moment we can approximate the amount of variance introduced to the model via the CIR perturbation \eqref{CIR process eqn}.
	
	In this setup we denote $u^{(2)}:=\mathbb{E}[I_t^2|I_0=x, Y_0=y]$ to emphasize the $\varphi(x)$ being a quadratic function of the initial condition. With this new notation, we are searching for a solution in the form $u^{(2)}(t,x,y)=\sum_{n\geq0}u^{(2)}_n(t,x,y)c^n$.
	
	When the variance of the perturbed solution is expressed with the first correction term as:
	\begin{align}
		\begin{split}
			Var(I_t) =& \mathbb{E}[I_t^2] - \mathbb{E}[I_t]^2,\\
			=&u^{(2)}_0+cu^{(2)}_1-\big(u^{(1)}_0+cu^{(1)}_1\big)^2+O(c^2),\\
			=&c\left( u^{(2)}_1 - 2u^{(1)}_0u^{(1)}_1 \right)+O(c^2).\\
		\end{split}
		\label{eqn variance correction terms}
	\end{align}

Due to $u^{(2)}:=(u^{(1)})^2$ by definition, so the only non-zero contributions will come from the correction terms. This makes sense because when $c=0$, then there is no perturbation, hence no variance should be present. Expression \eqref{eqn variance correction terms} provides how to evaluate the first correction term for the variance introduced to the system. Using the proposition \ref{proposition solving system correction terms}, the term $u^{(2)}_1$ is evaluated analytically in Wolfram using symbolic differentiation and integration.

Then the first variance correction term in \eqref{eqn variance correction terms} is evaluated and the result is plotted in figure \ref{figure var(I_t) correction}. 
\begin{figure}[h]
	\centering
	\includegraphics[width=0.7\textwidth]{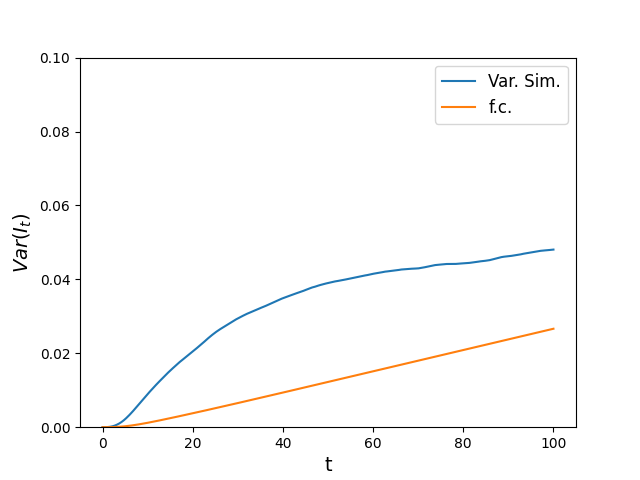}
	\caption{The figure showing the variance of the 1500 simulated sample paths in blue and the first correction term to the variance \eqref{eqn variance correction terms} in orange. The used parameters are: $c=0.1$, $x=0.3$, $y=b=0.5$, $\gamma=0.3$, $\sigma=0.063$, $a=0.02$, hence the infection persists and the perturbation is always positive.}
	\label{figure var(I_t) correction}
\end{figure}

One can see the first correction term for the $Var(I_t)$ was able to capture the increasing trend of variance dynamics of the solution to some extend. Of course with more correction terms considered, a better variance estimate of the model \eqref{perturbed SIS model} would be obtained.

\section{Comparison of Different Perturbation Types}\label{section perturbation comparisons}
Another important question to address in model perturbations is that how does different perturbation choices of the same parameter affect the model? Consider the perturbation performed in \cite{graySISpaper}, where the authors make the disease transmission coefficient perturb with an additive Brownian noise:
$$\beta dt\to\beta dt + \sigma dB_t.$$
and the CIR perturbation performed in this paper as:
$$\beta \to Y_t,\quad Y_0=y,$$
where $Y_t$ solves the SDE \eqref{generic Y_t perturbation} with parameters $y=b=\beta$. 
 
%
 
One can note that the expectation of both perturbations are constant and equal to $\beta$. The natural question arises: Do they have similar behavior for expected solution of the perturbed SIS model \eqref{perturbed SIS model}?

\subsection{The System of fkPDEs of Perturbation Performed on \cite{graySISpaper}}
In this section, the perturbation used in \cite{graySISpaper}  will be reorganized into the fkPDE approach to come up with the $0$'th and first order corrections. When the perturbation of \cite{graySISpaper} is applied on the deterministic SIS model the following SDE is obtained:
\begin{equation}\label{gray perturbed SIS model}
	dI_t = \left( \beta I_t(1-I_t) - \gamma I_t \right)dt + \tilde{\sigma}\sqrt{c} (I_t(1-I_t)) dB_t,\quad I_0=x,
\end{equation}
where the scaled constants satisfy the relations: $\tilde{\sigma}\sqrt{c} = \sigma$. One can define $g(t,x):=\mathbb{E}[\varphi(I_t)|I_0=x]$ and write the associated Feynman-Kac equation as:

\begin{equation}\label{eqn gray feynman-kac}
	\frac{\partial g(t,x)}{\partial t}=\Big( \beta x(1-x)-\gamma x \Big)\frac{\partial g(t,x)}{\partial x}+c\left( \frac{\tilde{\sigma}^2}{2}x(1-x) \right)\frac{\partial^2 g(t,x)}{\partial  x^2},\quad g(0,x)=\varphi(x).
\end{equation}



Suppose the solution to the fkPDE \eqref{eqn gray feynman-kac} is in the form 
$$g(t,x):=\sum_{n\geq0}g_n(t,x)c^n.$$

Similarly this assumption expands the single PDE to a system of PDEs:

\begin{align}\label{eqn system gray correction terms}
	\begin{split}
		\frac{\partial g_0}{\partial t}=\frac{\partial g_0}{\partial x}\Big( \beta &x(1-x) \Big),\\
		\frac{\partial g_1(t,x)}{\partial t}=\frac{\partial g_1}{\partial x}\Big( \beta x(1-x&) \Big)+\frac{\partial^2 g_0}{\partial x^2}\left( \frac{\tilde{\sigma}^2}{2}x(1-x) \right),\\
		\frac{\partial g_2(t,x)}{\partial t}=\frac{\partial g_2}{\partial x}\Big( \beta x(1-x&) \Big)+\frac{\partial^2 g_1}{\partial x^2}\left( \frac{\tilde{\sigma}^2}{2}x(1-x) \right),\\
		&\vdots
	\end{split}
\end{align}

\begin{proposition}\label{proposaition gray correction}
	The solution to the system \eqref{eqn system gray correction terms} is given as:
	
\begin{align*}
	g_n(t,x) &= \frac{\tilde{\sigma}^2}{2}\int_{0}^{t}\left[ \frac{\partial^2g_{n-1}(s, I^D_{t-s}(x))}{\partial x^2}\right]ds,& for \:n\geq1,\\
	g_0(t,x)&=\varphi(I^D_{t}(x)), &for\: n=0.
\end{align*}
%
\end{proposition}
\begin{proof}
	The proof is using the methods as the proof of \ref{proposition solving system correction terms}.
\end{proof}

\begin{remark}\label{remark u0 g0 equal}
	We have the identity $u_0(t,x,y)=g_0(t,x)$ when $y=b=\beta$. This is due to the fact that, both perturbations are applied on the same deterministic model, so the affect of perturbations are only apparent starting from the first order of correction terms.
\end{remark}
\subsection{Comparison of Gray et. al. Perturbation with the Natural Perturbation on the Expectation}
In this section we would like to compare how 2 perturbations having the same expectation may affect the expectation of the infected population. To this aim, the leading correction terms will be used to show the affect of the perturbation on the deterministic model. Choose $\varphi(x)=x$ and using the pipeline suggested in proposition \ref{proposaition gray correction}, the first correction term to the $g(t,x):=\mathbb{E}[I_t|I_0=x]$ is again analytically evaluated in Wolfram. When the results are plotted, the figure \ref{figure gray vs CIR} is obtained.


\begin{figure}[h]
	\centering
	\begin{subfigure}{0.48\textwidth}
		\centering
		\includegraphics[width=\linewidth]{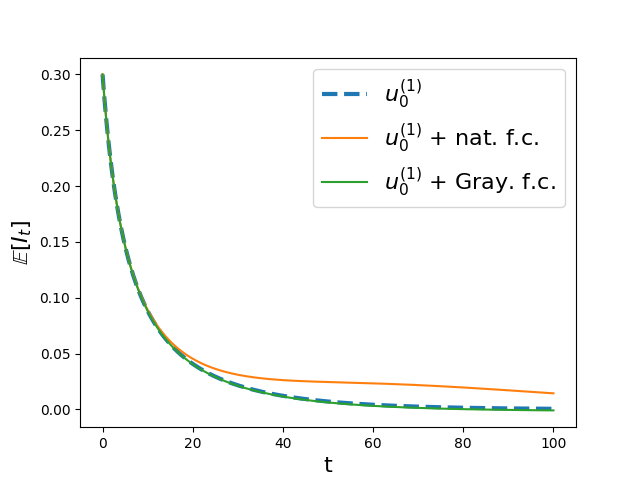}
		\caption{Extinction of the infection.}
	\end{subfigure}
	\hfill
	\begin{subfigure}{0.48\textwidth}
		\centering
		\includegraphics[width=\linewidth]{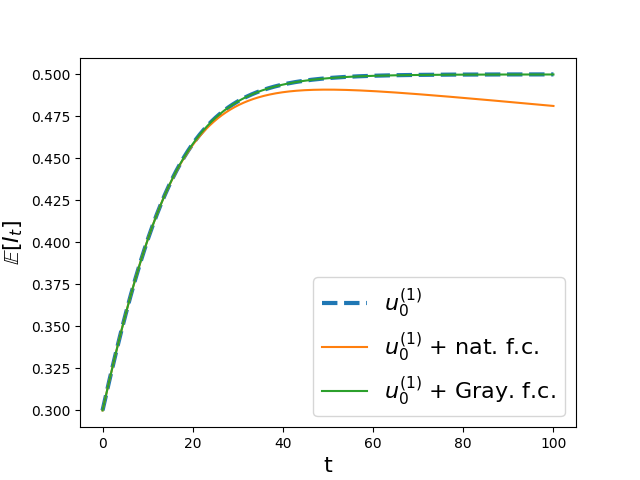}
		\caption{Persistence of the infection}
	\end{subfigure}
	\caption{The plot showing the deterministic solution $u^{(1)}_0(t,x,y)=g_0(t,x)$ plotted in blue dashed lines, the leading correction term for the Gray perturbation \eqref{gray perturbed SIS model} on the expectation in green solid line and leading correction term for the CIR perturbation \eqref{CIR process eqn} in orange solid line. The used parameter setup is the same as the figure \ref{figure average and correction terms}, (a): $c=0.1$, $x=0.3$, $y=b=0.45$, $\gamma=0.5$, $\sigma=0.063$, $a=0.02$ and (b): $c=0.1$, $x=0.3$, $y=b=0.2$, $\gamma=0.1$, $\sigma=0.032$, $a=0.02$.}
	\label{figure gray vs CIR}
\end{figure}

The figure \ref{figure gray vs CIR} highlights an important point for modelers, since when the first order corrections of the corresponding perturbations are considered, the resulting expected dynamics of infected population, deviate from each other. It means using two different perturbations having same expectations actually matter. It seems like the perturbation considered in Gray et. al. \cite{graySISpaper} doesn't affect the deterministic dynamics on average. However when the natural perturbation presented in this paper is considered with the CIR perturbation, one can see that the expectation of the model is more towards the center. 

This different dynamics of the various perturbation choices can inform the modeler and let them use the most suitable one for their needs. For example, a modeler aiming for a more "balanced" perturbation might consider the CIR model whereas another modeler interested in obtaining the deterministic dynamics might consider the perturbation performed in Gray's paper \cite{graySISpaper}.

\section{Discussion}
In this paper a natural perturbation approach is presented to the deterministic SIS model's disease transmission coefficient in section \ref{section SIS pertubration}. This new approach is natural in the sense that, it aims the same constraint on the disease transmission rate being non-negative. Even though the effect of violation of this constraint is hidden in the irregularities of the sample path of the solution $I_t$, still from a modeling point of view, it is an important issue. Moreover in section \ref{section SIS pertubration}, the complete analysis of the naturally perturbed model is presented: Boundedness of the solution and the conditions to have extinction and persistence of the infection. Lastly in section \ref{section diffusion examples} two one dimensional diffusion processes are presented as examples that satisfy the assumption \ref{assumptions}. Then the analysis performed on the natural stochastic SIS model is illustrated with figures for the CIR perturbation.

In section \ref{section on average analysis}, a technique to analytically approximate the expectation of any function of the solution of the stochastic SIS model, $\varphi(I_t)$ is presented. The technique basically relies on the applying an appropriately scaled perturbation on the deterministic SIS model and then converting the SDEs to the associated Feynman-Kac PDE (fkPDE). Lastly by using the scaled perturbation and the classical Perturbation Theory, one can expand the single fkPDE to a system of PDEs. It has been shown for $\varphi(x)=x$ that the first solution term of this system corresponds to the mean-field solution, which is the deterministic infected population. With increasing order of correction terms the approximation better approximates the actual expectation of the solution as presented in figure \ref{figure average and correction terms}. One should note that, the calculations done in this section can be generalized to a wider range of functions than $\varphi(x)=x$. As an example, $\varphi(x)=x^2$ is also considered and the first correction to the variance introduced to perturbed SIS model \eqref{perturbed SIS model} is shown on figure \ref{figure var(I_t) correction}.

Lastly using this technique relying on appropriately scaled fkPDE and Perturbation Theory, two different perturbations are compared in terms of their effect on the expectation of the infected population: The CIR model and \eqref{gray perturbation} from \cite{graySISpaper}, both having the same expectation. It turned out, their effect was different. Such differences are important from a modeling point of view because different perturbation methods with same expectation yield different on average dynamics. Luckily, the technique presented in this paper, modelers can analytically trace the effect of their desired perturbations and compare different models. This approach being able to generalized to other functions $\varphi(x)$ with the provided recursive way for evaluation of higher order corrections terms provide promising results.

\subsection*{Acknowledgements}
I would like to thank Prof. Alberto Lanconelli from University of Bologna, for the invaluable guidance and encouraging attitude.

	\newpage
	\bibliographystyle{plain}
	\bibliography{biblio}
	
\end{document}